\documentclass[reqno,10pt]{amsart}
\usepackage{mathrsfs,amssymb,amsmath,url}
\usepackage{verbatim}
\usepackage{color}
\usepackage{hyperref}
\usepackage{tikz}

\newtheorem{Theorem}{Theorem}[section]
\newtheorem{Proposition}[Theorem]{Proposition}
\newtheorem{Lemma}[Theorem]{Lemma}

\newtheorem{Conjecture}[Theorem]{Conjecture}

\newtheorem*{claim}{Claim}

\theoremstyle{remark}
\newtheorem{Remark}[Theorem]{Remark}

\newcommand*{\fplus}{\genfrac{}{}{0pt}{}{}{+}}
\newcommand*{\fdots}{\genfrac{}{}{0pt}{}{}{\cdots}}
\newcommand*{\fminus}{\genfrac{}{}{0pt}{}{}{-}}

\newmuskip\pFqskip
\mathchardef\pFcomma=\mathcode`, 
\newcommand*\pFq[5]{%
  \begingroup
  \begingroup\lccode`~=`,
    \lowercase{\endgroup\def~}{\pFcomma\mkern\pFqskip}%
  \mathcode`,=\string"8000
  {}_{#1}F_{#2}\biggl(\genfrac..{0pt}{}{#3}{#4};#5\biggr)%
  \endgroup
}

\newcommand{\rfac}[2]{{\left({#1}\right)_{#2}}}

\hypersetup{
    colorlinks=true, 
    linktoc=all,     
    linkcolor=blue,  
    citecolor = blue
}
\allowdisplaybreaks

\numberwithin{equation}{section}

\dedicatory{Dedicated to the memory of Dick Askey}

\begin{document}
\title[Telescoping Continued fractions]{Telescoping continued fractions for the error term in Stirling's formula}

\author[G.~Bhatnagar]{Gaurav Bhatnagar
}
\address{Ashoka University, Sonipat, Haryana, India}
\email{bhatnagarg@gmail.com}

\author[K.~Rajkumar]{Krishnan Rajkumar}
\address{Jawaharlal Nehru University,
Delhi, India.}
\email{krishnan.rjkmr@gmail.com}

\begin{abstract}
In this paper, we introduce telescoping continued fractions to find lower bounds for the error term $r_n$ in Stirling's approximation
$\displaystyle n! = \sqrt{2\pi}n^{n+1/2}e^{-n}e^{r_n}.$
 This improves lower bounds given earlier by Ces\`{a}ro (1922), Robbins (1955),  Nanjundiah (1959), Maria (1965) and Popov (2017).
The expression is in terms of a  continued fraction, together with an algorithm to find successive terms of this continued fraction. The technique we introduce allows us to experimentally obtain
upper and lower bounds for a sequence of convergents of a continued fraction
 in terms of a difference of two continued fractions. 

\end{abstract}


\maketitle

\section{Introduction}
To prove Stirling's approximation for $n!$, one approach is to 
show 
$$n! = \sqrt{2\pi}n^{n+1/2}e^{-n}e^{r_n},$$
where $r_n$ satisfies an inequality of the form
$$\frac{1}{12n+a(n)} \le r_n \le \frac{1}{12n}.$$
Stromberg \cite[p.~253]{Stromberg1981} outlines a proof, and says Ces\`{a}ro (1922) obtained the estimate $a(n)=1/4$.  Robbins' \cite{Robbins1955} gave $a(n)=1$ and was most influential; he takes $\epsilon_p$, for integer $p \geq 1$, such that 
$$r_n = \sum_{p\ge n} \epsilon_p.$$
Here 
\begin{align}\label{first}
\epsilon_p = \frac{2p+1}{2} \log \left(\frac{p+1}{p}\right) -1 
			=\frac{1}{2y} \log \left(\frac{1+y}{1-y}\right) -1,
\end{align}
where $y = 1/(2p+1)$.
The key idea of Robbins' proof is to use the 
Taylor expansion of \eqref{first} about $y=0$, to conclude that
\begin{equation}\label{robbins}
\frac{1}{12} \left( \frac{1}{p+\frac{1}{12}}-\frac{1}{p+1+\frac{1}{12}}\right) < \epsilon_p < 
\frac{1}{12} \left( \frac{1}{p}-\frac{1}{p+1}\right).
\end{equation}
Both bounds are of the form $g(p)-g(p+1).$ 
The estimates for $r_n$ follow by telescoping. 

Robbins' proof has motivated further enhancements, in particular by Maria~\cite{Maria1965},
Nanjundiah \cite{Nanjundiah1959}, Hirschhorn and Villarino~\cite{HV2014}, and Popov~\cite{Popov2018}. All these authors work with the Taylor expansion of \eqref{first}, just like Robbins. Dominici~\cite{Dominici2008} has written an interesting history of approaches to prove Stirling's approximation.

We remark that finding upper and lower bounds for $r_n$ are not the main item of interest here. 
Indeed, $r_n$ (also called the Binet function and denoted by $J(n)$)
 possesses the asymptotic expansion 
\begin{equation}\label{stirling}
	r_n \sim \sum_{i=1}^{\infty} \frac{B_{2i}}{(2i-1)2i}n^{-2i+1},
\end{equation}
called the Stirling series in the literature. Explicit integral expressions are known for the remainder term for any finite truncation of \eqref{stirling} and the sign of this term is the same as that of the first neglected term in the asymptotic expansion. Thus the finite truncations of \eqref{stirling} of length $N$ give upper and lower bounds for $r_n$ depending on the parity of $N$. However, these bounds diverge rapidly, and the above cited authors have provided bounds which work for all positive integers $n$. 

The goal of this paper is to take Robbins' proof in a different direction.  What is different is that we use the continued fraction \eqref{cf1}, rather than the Taylor expansion of \eqref{first}. 
The convergents of this continued fraction can be written in terms of the Legendre polynomials and the associated (numerator) polynomials.  We find a sequence of lower and upper estimates of the
form $g_m(p)-g_m(p+1)$
for these convergents. By telescoping we obtain lower bounds for $r_n$ which better the bounds found by the above mentioned authors.


Moreover, our function $g_m(p)$ is a continued fraction of the form
\begin{equation}\label{gm}
g_m(p)=
\frac{a_1}{p}\fplus\frac{a_2}{p}\fplus\frac{a_3}{p}\fplus\fdots\fplus\frac{a_m}{p}.
\end{equation}
Thus, we call this approach the method of telescoping continued fractions. (The idea of telescoping to discover continued fractions is implicit in \cite{Rajkumar2012}.)



This paper is organized as follows. We provide an overview of our technique in \S\ref{sec:overview}. In \S\ref{sec:lower-bounds}, we apply this technique and obtain lower bounds for the error function. Next, in \S\ref{sec:summary-calc} we outline what is required to make this approach work, and prove the main algorithm in \S\ref{sec:algo}. 
In \S\ref{sec:conjectures} we discuss the results and 
and note some conjectures related to our findings. 
We conclude with the motivation of this paper---a Ramanujan story.

\section{Overview of the technique}\label{sec:overview}
We wish to estimate 
$$\epsilon_p = \frac{2p+1}{2} \log \left(\frac{p+1}{p}\right) -1.$$
We use the continued fraction
\begin{equation}\label{cf1} 
\frac{1}{2y} \log \left(\frac{1+y}{1-y}\right)= \frac{1}{1}\fminus \frac{y^2}{3}\fminus \frac{4y^2}{5}\fminus \frac{9y^2}{7}\fminus \fdots,
\end{equation}
where $y=1/(2p+1)$, or equivalently,
\begin{equation}\label{cf}
\frac{1}{2}\log \left(\frac{p+1}{p}\right)= \frac{1}{2p+1}\fminus \frac{1}{6p+3}\fminus \frac{4}{10p+5}\fminus \frac{9}{14p+7}\fminus \fdots.
\end{equation}

Let  $$p_k(x)=x\frac{P_k^*(x)}{P_k(x)}-1,$$ 
where $P_n(x)$ and $P^*_n(x)$ are the Legendre Polynomials and their associated numerator polynomials. Our approach consists of the following steps.
\begin{description}
\item[Step 1] We show that $p_k(2p+1)$ is an increasing sequence of approximations to $\epsilon_p$. 
\item[Step 2] Let $g_m(p)$ be given by \eqref{gm}. For fixed $k$, we solve a series of inequalities to find, in turn, $a_1$, $a_2$, $\dots$, $a_m$, so that 
 upper and lower estimates for
$p_k(2p+1)$ are provided by alternate terms of 
$g_1(p)-g_1(p+1)$, $g_2(p)-g_2(p+1)$, $\dots$, 
$g_m(p)-g_m(p+1)$.
\end{description}
By Step 1, it follows that lower bounds for $p_k(2p+1)$ are also lower bounds for 
$\epsilon_p$. Thus, by summing over $p\ge n$, we find lower bounds for 
 $r_n$ of the form
\begin{equation}\label{rn_upper}
r_n> \frac{a_1}{n}\fplus\frac{a_2}{n}\fplus\fdots\fplus\frac{a_m}{n}, 
\end{equation}
for $m$ an even number. (Upper bounds will be discussed later in $\S$\ref{sec:summary-calc}) 

We first prove our claims in Step 1. The convergents of this continued fraction form an increasing sequence. This is true of any continued fraction of the form
$$\frac{a_1}{b_1}\fminus \frac{a_2}{b_2}\fminus \frac{a_3}{b_3}\fminus\fdots, $$    
where the $(a_k)$ and $(b_k)$ are positive. 
Next, we show that the convergents of the continued fraction in \eqref{cf}  are given by
${P^*_n(2p+1)}/{P_n(2p+1)}$. This is a consequence of the following Proposition.
\begin{Proposition}\label{prop:1} Consider the continued fraction.
\begin{equation*}
\frac{1}{x}\fminus\frac{1^2}{3x}\fminus\frac{2^2}{5x}\fminus\frac{3^2}{7x}\fminus\fdots.
\end{equation*}
The $k$th convergent of this continued fraction is given by $P_k^*(x)/P_k(x)$. The sequence 
$P_k^*(x)/P_k(x)$ converges to the continued fraction for $x\not\in [-1,1]$. 
\end{Proposition}
\begin{proof}[Outline of proof]
Let the  $k$th convergent of the continued fraction be given by ${N_k(x)}/{D_k(x)}$.
Let
$$P^*_k(x) := \frac{N_k(x)}{k!} \text{ and } P_k(x) := \frac{D_k(x)}{k!}.$$
We refer to Theorem~2.6.1 of Ismail~\cite[p.\ 35]{MI2009} to see that 
$P^*_k(x)$ and $P_k(x)$ satisfy the recurrence relation
\begin{equation} \label{three-term-legendre}
y_{k+1}(x) = \frac{2k+1}{k+1} x y_k(x) - \frac{k}{k+1} y_{k-1}(x), \text{ for } k>0,
\end{equation}
with the initial values 
$P_0(x)=1, P_1(x)=x, P^*_0(x)=0, P^*_1(x)=1.$
 Thus  $P_k(x)$ and $P^*_k(x)$ are the Legendre polynomials and the associated numerator polynomials (see Chihara \cite[p.\ 201]{Chihara1978}).
The convergence follows from the theory of orthogonal polynomials, see \cite[Th.~2.6.2, p.~36]{MI2009}. The convergents converge to the continued fraction for all  $x\in \mathbb{C}$, where $x\not\in [-1,1]$.
\end{proof}

Before proceeding with more details of Step 2, note one interesting property that helps to simplify the computations. It turns out that both sides of the inequalities we wish to solve are functions of $z=p(p+1)$. First we consider $p_k(x)$. 

\begin{Proposition}\label{prop:2} Let $p_k(x)$ be as defined above. Let $z=p(p+1)$.Then
for all $p$ and $k\geq 1$, $p_k(2p+1)$ is a rational function of $z$.
\end{Proposition}
\begin{proof} 
First note that for $k>1$, $p_k(x)$ is an even function. This follows from
$$
P_k^*(-x) =(-1)^{k-1}P_k^*(x) \text{ and }
P_k(-x) =(-1)^{k}P_k(x).
$$
These are easy to prove by induction using the recurrence relation \eqref{three-term-legendre}.
Now, since $p_k(x)$ is an even rational function, it must be a rational function in $x^2$. Note that $(2p+1)^2=4p(p+1)+1$. Thus $p_k(2p+1)$ is a function of $z=p(p+1)$. 
\end{proof}

Let $g_m(p)={n_m(p)}/{d_m(p)}$. We use the simplified notation $n_j:= n_j(p),d_j:=d_j(p), n_j^+:= n_j(p+1),d_j^+:= d_j(p+1).$ 
\begin{Lemma}\label{prop:hz} 
Let $z=p(p+1)$.Then
for all $m$ and $p$, $g_m(p)-g_m(p+1)$ is a rational function of $z$. That is:
\begin{enumerate}
\item[(a)] the numerator of $g_m(p)-g_m(p+1)$ given by $n_m d_m^+ - n_m^+ d_m$ is
a polynomial in $z$; and,
\item[(b)] the denominator $d_m d_m^+$ is
a polynomial in $z$.
\end{enumerate}
\end{Lemma}
\begin{proof} First note that $n_k$, and $d_k$ satisfy the recurrence 
\begin{equation}\label{g-n-d-recursion}
u_k=p u_{k-1}+a_k u_{k-2}, \text{ for } k\geq 1,
\end{equation}
with initial values
$n_{-1}=1,d_{-1}=0,n_0=0,d_0=1.$

We first prove Part (b) of the lemma by induction. Suppose (b) is true for $m=k$. Since $d_k^+$ also satisfies 
\eqref{g-n-d-recursion} (with $p\mapsto p+1$), $d_{k+1}d_{k+1}^+$ can be written as
\begin{equation*}
	d_{k+1}d_{k+1}^+ = p(p+1) d_k d_k^+ + a_{k+1}^2 d_{k-1} d_{k-1}^+
	+ a_{k+1} \left( p d_{k-1}^+ d_k  +(p+1) d_{k-1} d_k^+  \right).
\end{equation*}
Thus to show that
\begin{subequations}
\begin{equation}\label{denom1}
d_m d_m^+
\end{equation}
is a polynomial in $z$, we need to show that 
\begin{equation}\label{denom2}
pd_{m-1}^+  d_m  +(p+1) d_{m-1} d_m^+ 
\end{equation}
is also a polynomial in $z$. Again, to use induction, we will need to consider when $m=k+1$. For this we need to show that
\begin{equation}\label{denom3}
d_{m-1}^+ d_m  - d_{m-1} d_m^+ 
\end{equation}
too is a polynomial in $z$. 

\end{subequations}
We can now complete the proof of part (b). All of \eqref{denom1}, \eqref{denom2} and \eqref{denom3}
are polynomials of degree zero when  $m=0$. For $m>1$, the result follows by induction.
%
%
%
The proof of part (a) is similar and is omitted.
\end{proof}

We are now ready to apply the technique and calculate the $a_m$ for small values of $k$. Even these will be enough to improve on previously known lower bounds. 

\section{Some lower bounds for the error function}\label{sec:lower-bounds}
In this section we use our technique to improve known lower bounds for $r_n$ given by Maria~\cite{Maria1965},
Nanjundiah \cite{Nanjundiah1959}, and Popov~\cite{Popov2018}, in addition to the aforementioned bounds due to Ces\`{a}ro and Robbins. 
We now write $p_k(2p+1)$ as functions of $z$ using the notation ${f_k^*(z)}/{f_k(z)}$ (see Proposition~\ref{prop:2}). We proceed with the calculations indicated in Step 2 of \S\ref{sec:overview} for $k=2, 3, 4, 5$. 


\subsection*{Estimates with $k=2$}
Let $k=2$. We need to find $a_1$ such that
$$\frac{f_2^*(z)}{f_2(z)} =  \frac{1}{12z + 2} <  g_1(p)-g_1(p+1)=\frac{a_1}{z}.$$
Clearing denominators, we see that we require $a_1$ such that
$(1-12a_1)z-2a_1< 0.$
The choice of $a_1=1/12$ works; the left hand side is negative for all values of $z\ge 1$. 
The next inequality is
\begin{multline*}
\frac{f_2^*(z)}{f_2(z)} > g_2(p)-g_2(p+1) = 
\frac{1/12}{p}\fplus\frac{a_2}{p}- \frac{1/12}{p+1}\fplus\frac{a_2}{p+1}
=
\frac{(1/12)(z-a_2) }{z^2+2a_2z +a_2^2 +  a_2},
\end{multline*}
where we have replaced $a_1$ by $1/12$. 
 Again, upon clearing the denominators, 
we see that we need to choose $a_2$ such that
%
$\big(3a_2-{1}/{6}\big)z + a_2^2 +  {7a_2}/{6} > 0.$
The choice of $a_2=1/18$ makes sure the expression is positive for all values of $z$. 

So far we have obtained the lower bound 
\begin{equation}\label{lb1}
r_n> \frac{1/12}{n}\fplus\frac{1/18}{n}.
\end{equation}
This lower bound is already better than 
Robbins' bound
$$r_n > \frac{1/12}{n}\fplus\frac{1}{12},$$
and, for $n>3$, Ces\`{a}ro's bound
$$r_n > \frac{1/12}{n}\fplus\frac{1}{48}.$$
Further, for $n>4$, the bound \eqref{lb1} improves Maria's~\cite{Maria1965} lower bound
$$r_n > \frac{1/12}{n}\fplus\frac{1}{16n+8}.$$

\subsection*{Estimates with $k=3$}
The next convergent is given by
$$ \frac{f_3^*(z)}{f_3(z)}  = \frac{5}{60z + 6}.$$
We have to solve the inequality
$$\frac{f_3^*(z)}{f_3(z)} \;?\; g_m(p)-g_m(p+1)$$
where $?$ represents $<$ (or $>$) depending on whether $m$ is odd (respectively, even). As before, let us clear denominators and bring all expressions to the left. Let $\Delta_1(3, m)$ denote the result obtained by doing so. We assume that at each step we have substituted the values of $a_1$, $a_2$, $\dots$, $a_{m-1}$ found previously. 

%

The calculated values of the first four 
terms is shown in Table~\ref{table:k3_estimates}.
\renewcommand{\arraystretch}{1.5}

\begin{table}[h]
$$
\begin{array}{| c | c | c | c |}
\hline
m & \Delta_1(3, m) & a_m & \text{sign}\\ \hline
1 & (5-60a_1)z - 6a_1 & a_1=\frac{1}{12} & < 0\\ \hline
2 & (15a_2 - 1/2)z + 5a_2^2 +  11a_2 /2&  a_2=\frac{1}{30} & > 0\\ \hline
3& (17/90 - 5a_3/6)z  - a_3^2/2 - 31a_3/60 & a_3=\frac{17}{75} & < 0 \\ \hline
4& \big(119a_4/450  - 357/2500\big)z + 17a_4^2/90 +  1037a_4/4500 & a_4=\frac{27}{50} & > 0 \\ \hline
\end{array}
$$
\label{table:k3_estimates}
\caption{Solving the inequality with $k=3$}
\end{table}
Note that we obtain the value of $a_i$ from the coefficient of $z$. We have to verify that the sign of the remaining terms is what we require. If the sign is negative we get an upper bound for 
$f^*_3(z)/f_3(z)$; otherwise a lower bound. 

We have found two more lower bounds, from the
second and fourth convergent. From the second convergent, we get
\begin{equation}\label{lb2}
r_n> \frac{1/12}{n}\fplus\frac{1/30}{n}.
\end{equation}
This improves \eqref{lb1} and the lower bound given by
Nanjundiah \cite{Nanjundiah1959}: $$r_n> \frac{{1/12}}{n}\fplus \frac{n}{30n^2-1}.$$ 

\subsection*{Estimates with $k=5$} As $k$ becomes bigger, the inequality we need to solve has higher powers of $z$. Nevertheless, we can still find $a_1$, $a_2$, $\dots$, from the coefficient of the highest power. 
For $k=5$, we find that
$$ \frac{f_5^*(z)}{f_5(z)}  = \frac{630z + 77}{ 7560z^2 + 1680z + 60}.$$
The expressions corresponding to $\Delta_1(5, m)$  
are as follows. 
\begin{align*}
m=1: \; & \big(630-7560a_1\big)z^2 - (1680a_1- 77)z  - 60a_1 \\ 
m=2: \; & \big(1890a_2 - 63\big)z^2   + \big(630a_2^2  + 924a_2  - 5\big)z + 77a_2^2 + 82a_2  \\
m=3:\; 
&  \big(53/2- 105a_3\big)z^2 -\big(63a_3^2  +1088a_3/15 - 2537/900\big)z  - 5a_3^2 - 31a_3/6
\\
m=4:\;
& \big(371a_4/10 - 39/2\big)z^2
+\big( 53a_4^2/2 +11777a_4/315   - 159/98\big)z
 \\
&\hspace{2.5in}
+ 2537a_4^2/900 + 4421a_4/1260 
\\
\end{align*}
The values we obtain are:
$$a_1=\frac{630}{7560}= \frac{1}{12}; a_2=\frac{63}{1890}= \frac{1}{30}; 
a_3=\frac{53/2}{105}= \frac{53}{210}; a_4=\frac{39/2}{371/10}= \frac{195}{371}.$$

Now at each stage we have to verify that the remaining terms have the correct sign. For example, for $m=2$, on replacing $a_2$ by $1/30$, we find that
$$\big(630a_2^2  + 924a_2  - 5\big)z + 77a_2^2 + 82a_2 = 
53z/2 + 2537/900,$$
 which is positive for all positive $z$. In this manner, we obtain the lower bound
\begin{equation}\label{lb3}
r_n> \frac{1/12}{n}\fplus\frac{1/30}{n}\fplus\frac{53/210}{n}\fplus\frac{195/371}{n}.
\end{equation}
This improves the lower bound given by Popov \cite{Popov2018}, who showed
$$r_n>\frac{1/12}{n}\fplus\frac{1/30}{n}\bigg(1-\frac{1/4}{(n+0.5)^2}\bigg).$$
But \eqref{lb3} is better; this follows from the inequality
$$1-\frac{1/4}{(n+0.5)^2} > \frac{1}{1}\fplus\frac{53/210}{n^2+195/371}.$$

To summarize, we take successive convergents of the continued fraction for $\epsilon_p$, and find continued fractions of the form $g_m(p)$ such that $g_m(p)-g_m(p+1)$ give estimates for these convergents. The lower bound estimates telescope to give lower bounds of the form \eqref{rn_upper} for $r_n$ and improve upon those obtained by Cesaro, Robbins, Maria, Nanjundiah and Popov. Further lower bounds of the form \eqref{rn_upper} can be obtained from the values of $a_m$ reported in Table~\ref{table:am_small_valuesk}. 

\begin{table}[h]
$$
\begin{array}{| c | c | c | c | c | c | c |  }
\hline
k & a_1 & a_2 & a_3 & a_4 & a_ 5 & a_6  \\ \hline\hline
2 & 1/12 & 1/18 & 11/45 & 39/70 & 188/189 & 925/594 
\\ \hline
3 & 1/12 &
1/30 &
17/75 &
27/50 &
44/45 &
305/198 
\\ \hline
4 & 1/12 &
1/30 &
53/210 &
1377/2597 &
1198100/1192023 &
80881615183/52930560375 
\\ \hline
5 & 1/12 &
1/30 &
53/210 &
195/371 &
56428/55809 &
248094749/163401381 
\\ \hline
6 & 1/12 &
1/30 &
53/210 &
195/371 &
22999/22737 &
329394523/217064562 
\\ \hline
7 & 1/12 &
1/30 &
53/210 &
195/371 &
22999/22737 &
29944523/19733142 
\\ \hline
\end{array}
$$
\caption{Values of $a_m$ for small values of $k$}
\label{table:am_small_valuesk}
\end{table}

\section{Summary of calculations}\label{sec:summary-calc}
The calculations in \S\ref{sec:lower-bounds} can be generalized.  Let $f^*_k(z)$, $f_k(z)$, $g_m(p)$ be as before. Recall from Proposition~\ref{prop:hz} that $g_m(p)-g_m(p+1)$ is a rational function in $z=p(p+1)$. We now suppress the dependence on $z$, and use 
$h_m^*$ and $h_m$ to denote the numerator and denominator of this difference. In terms of our earlier notation, we have
$$\frac{h_m^*}{h_m}:=g_m(p)-g_m(p+1)=\frac{n_m}{d_m}-\frac{n_m^+}{d_m^+}.$$
We also used $\Delta_1(k,m)$ which is defined as
$$\Delta_1(k,m):= f^*_kh_m -f_k h^*_m.$$
where we now use $f^*_k=f^*_k(z)$ and $f_k=f_k(z)$. 
This is obtained by clearing out the denominators in the inequality we wish to find, and bringing all terms to the left hand side. 

In our examples in \S\ref{sec:lower-bounds}, 
for a fixed $k\ge 2$, we considered in turn $\Delta_1(k,m)$ for $m=1, 2, 3, \dots.$
The following was observed.
For each $m$, we can find $a_m$ by taking the value that makes the coefficient of the highest degree term of $\Delta_1(k,m)$ equal to $0$. Further, in our examples, we observe the following.
\subsection*{Observations}
Let $a_m$ and $\Delta_1(k,m)$ be as above. Then, 
\begin{itemize}
\item the values $a_m>0$ for $m=1, 2, \dots$; and,
\item the sign of the remaining terms (once $a_m$ is chosen) is given by $(-1)^m$; that is, the direction of the inequality is pre-determined. 
\end{itemize}

Assuming this happens, it is clear that if we obtain a lower bound for $f_k(p)$ then it is also a lower bound for $\epsilon_p$ and thus gives a lower bound for $r_n$. 

There is one further useful observation from 
Table~\ref{table:am_small_valuesk}. The values of $a_m$ appear to stabilize below the diagonal; that is, for a fixed $m$, the values of $a_m$ obtained from $\Delta(k, m)$ are the same for  $k\ge m+1$. 

This observation is relevant to finding upper bounds for $r_n$, in addition to lower bounds, still under the assumption that the observations noted above are true. 
Because  if $f_k^*(2p+1)/f_k(2p+1)<g_m(p)-g_m(p+1)$ then the RHS is an upper bound for $\epsilon_p$ by taking $k\to \infty$ (see Proposition~\ref{prop:1}).

Thus assuming the above, it follows that upper (and lower) bounds for $\epsilon_p$ are given by  $g_m(p)-g_m(p+1)$ for $m$ odd (respectively, even). 
By telescoping, it follows that $r_n$ is bounded above  and below by $g_m(n)$ for $m$ odd and even, respectively, for $m=1, 2, 3, \dots.$ 

In the next section, we formalize the ``algorithm'' described above and prove that we can always find the $a_m$ in Theorem~\ref{thm:degrees}.

We have been able to prove only some of our observations (the proofs are not included here). In particular, we have been unable to prove the sign changes in the second item of the observations. Nevertheless, for a specific $k$, we can compute $a_m$ using the algorithm;  in case it turns out that the direction of the inequalities are all correct, and if $m$ is even---we obtain lower bounds for $r_n$.  


\section{The algorithm}\label{sec:algo}

In this section, 
we prove the algorithm used to calculate the continued fraction  $g_m(p)$ for $m=1, 2, \dots$. 

Along with $\Delta_1(k,m)$ we will require three additional 
polynomials.  
\begin{align*}
\Delta_1(k,m) &:= f_k^*(d_m d_m^+)-f_k(n_m d_m^+- n_m^+ d_m) \\
\Delta_2(k,m) &:= f_k^*(p d_m d_{m-1}^+ + (p+1) d_m^+ d_{m-1})\\
&\hspace{1.cm} -f_k(p n_m d_{m-1}^+ - (p+1)n_m^+ d_{m-1}+(p+1) n_{m-1} d_m^+ - p n_{m-1}^+ d_m)\\
\Delta_3(k,m) &:= f_k^*(d_m d_{m-1}^+ - d_m^+ d_{m-1})\\
&\hspace{1cm}-f_k(n_m d_{m-1}^+ + n_m^+ d_{m-1}-n_{m-1}^+ d_m -n_{m-1} d_m^+) \\
\Delta_4(k,m) &:= \Delta_2(k,m)+\Delta_3(k,m).
\end{align*}
We have suppressed the dependence on $z$ and $a_1$, $a_2$, $\dots$, $a_m$  in our notation above.  These polynomials can be computed recursively. 

From the initial conditions for $n_m$ and $d_m$, we have the initial conditions
$$\Delta_1(k,-1) = 0,  \; \Delta_1(k,0) = f_k^*;  
\Delta_2(k,0) = -f_k; \; \Delta_3(k,0) = 2f_k;\;
\Delta_4(k,0)=f_k.
$$

\begin{Lemma}\label{lemma:delta} For fixed $k>1$, and for $m=1, 2, \dots$, 
 we have the following relations.
\begin{subequations}
\begin{align}
\Delta_1(k,m+1) &= z \Delta_1(k,m) + a_{m+1} \Delta_2(k,m) +a_{m+1}^2 \Delta_1(k,m-1)
\label{delta1}; \\
\Delta_2(k,m+1) &= (2z+1) \Delta_1(k,m) + a_{m+1} \Delta_2(k,m)+ a_{m+1} \Delta_3(k,m)
\label{delta2}; \\
\Delta_3(k,m+1) &= - \Delta_1(k,m) - a_{m+1} \Delta_3(k,m); \label{delta3}\\
\Delta_4(k,m+1) &= 2z \Delta_1(k,m)+a_{m+1}\Delta_2(k,m). \label{delta4}
\end{align}
\end{subequations}

\end{Lemma} 
\begin{proof} The proof is similar to that of Lemma~\ref{prop:hz}. Some of the polynomials mentioned in Lemma~\ref{prop:hz} appear in the definitions of $\Delta_1(k,m)$, $\Delta_2(k,m)$ and $\Delta_3(k,m)$. Relations such as those given in the proof of
Lemma~\ref{prop:hz} are required for the proof. 
\end{proof}


Next, we show how to compute $a_m$, for $m=1, 2, 3, \dots$. First we note that degree (in $z$) of $f_k^*$ is one less than that of $f_k$. 

\begin{Theorem}\label{thm:degrees} Let $k\ge 2$ be fixed and let  $d_k>0$ denote the degree of the polynomial $f_k$. Suppose the degree of $f_k^*$ is $d_k-1$.  Then there is a unique choice of $a_1$, $a_2$,  $\dots$, $a_m$
which make the degree (in $z$) of the polynomials $\Delta_1(k,m)$, $\Delta_2(k,m)$, and, $\Delta_3(k,m)$  equal to
 $d_k-1$, $d_k$, and, $d_k$, respectively.
\end{Theorem}
\begin{proof} The proof follows by induction from the relations in Lemma~\ref{lemma:delta}. At each step, we use \eqref{delta1} in the form
\begin{equation}\label{def:am}
\Delta_1(k,m) = z \Delta_1(k,m-1) + a_{m} \Delta_2(k,m-1) +a_{m}^2 \Delta_1(k,m-2).
\end{equation}

For $m=1$, we require degree of $f_k$ given by $d_k$ is greater than $0$, and the degree of $f_k^*$ is $d_k-1$. 

By induction, after choosing $a_1$, $a_2$, $\dots$, $a_{m-1}$, the degree of the polynomials
$z \Delta_1(k,m-1)$ and $a_{m} \Delta_2(k,m-1)$ is $d_k$.
Now we can choose  $a_{m}$ in such a manner that the highest degree term of the first two terms is $0$, and so the degree of the left hand side becomes $d_k-1$ (see Remarks  \ref{remark:termination1} 
for more details on this claim).  Note that the degree of $a_{m}^2 \Delta_1(k,m-2)$ is $d_k-1$ by induction, and thus does not influence the choice of $a_m$. 
\end{proof}
\begin{Remark}\label{remark:termination1} Lemma~\ref{lemma:delta} and Theorem~\ref{thm:degrees} remain true  if $f_k^*$ and $f_k$ are replaced by arbitrary polynomials of degree $d_k-1$ and $d_k$ respectively, with an important exception---while the choice of $a_m$ is unique at any stage to make the degree of $\Delta_1(k,m)$ to be $\leq d_k-1$, the resulting degree of $\Delta_1(k,m)$ could become $< d_k-1$. This would result in the situation of having no suitable choice of $a_{m+1}$ at the next stage. Thus we will need to terminate the algorithm should this situation arise.
\end{Remark}

\section{Discussion}\label{sec:conjectures}

We have only been partially successful in obtaining upper and lower bounds by our technique, because we have been unable to prove all our observations. In that respect, the situation is similar to the path taken by others mentioned in the introduction, who enhanced Robbins' proof, while using a Taylor expansion of \eqref{first}.  In this section we discuss our conjecture along with a 
further conjecture arising from our observation that the $a_m$ obtained by our technique appears to match those of the $S$-fraction giving the diagonal Pad\'e approximants to the Stirling series \eqref{stirling}.


Based on numerical evidence partially described in \S\ref{sec:lower-bounds}, we have the following conjecture.
\begin{Conjecture}\label{conj1b} Let $g_m(n)$ be defined as in \eqref{gm}, with $a_m$ be as obtained from Theorem~\ref{thm:degrees}.
The $g_m(n)$ give lower bounds ($m$ even) and upper bounds (respectively, $m$ odd) for $r_n$, that is,
$$r_n \text{ is }
\begin{cases}
\le \;\; g_m(n), & m \text{ even};\\
\ge \;\; g_m(n), & m \text{ odd}.
\end{cases}
$$
\end{Conjecture}

From Table~\ref{table:am_small_valuesk}, we see that $g_m(p)$ is the continued fraction which begins
\begin{equation}\label{cf-asymp-expansion}
\frac{{1/12}}{n}\fplus \frac{{1/30}}{n}\fplus 
\frac{{53/210}}{n}\fplus \frac{{195/371}}{n}\fplus\frac{{22999/22737}}{n}\fplus\fdots.
\end{equation}

By following the path described in this paper, we expected to find upper and lower bounds for $r_n$. To our surprise, we found 
that these terms match a continued fraction mentioned by Jones and Thron \cite[p.~350]{JT1980}, which is obtained from the asymptotic expansion  given in 
\eqref{stirling} by using the quotient-difference algorithm. This algorithm, as mentioned earlier, expresses the diagonal Pad\'e approximants in the form of a $S$-fraction.

For the sake of notation, let us denote the continued fraction in 
\cite[p.~350]{JT1980} as $B(n)$ which is also of the form
\begin{equation}\label{sfraction-n}
	B(n)=\frac{b_1}{n} \fplus \frac{b_2}{n}\fplus\frac{b_3}{n }\fplus\fdots
\end{equation}

Thus we have the following conjecture. 
\begin{Conjecture}\label{conj2} Let $a_m$ be the sequence of numbers generated by Theorem~\ref{thm:degrees}. The continued fraction $g_m(n)$ (as $m\to \infty$) is the same as $B(n)$, i.e., the sequence of numerators  $a_m=b_m$. 
\end{Conjecture}
The following facts are known about $B(n)$. 
\begin{enumerate}
\item  The sequence $b_m$ appears in a Stieltjes-type continued fraction of the form  
\begin{equation}\label{sfraction}
	\frac{b_1}{x} \fplus \frac{b_2}{1}\fplus\frac{b_3}{x }\fplus\fdots
\end{equation}
corresponding to a solution of a classical Stieltjes moment problem. See \cite[Theorem~5.1.1]{CPVWJ2008} for the basis of this relationship and 
\cite[\S 12.2, p.~224-225]{CPVWJ2008} for specific remarks on $J(z)$ which is called the Binet Function in this reference. From this it is evident that
$b_m>0$, and \eqref{sfraction} converges to $J(\sqrt{z})/\sqrt{z}$ in $|\arg(z)|<\pi$.
\item The sequence $b_m$ is obtained by using the quotient-difference algorithm from
the asymptotic series for $J(z)$; see \cite[p.~227]{JT1980}.
\item By a simple transformation, we see that $J(n)$ has a convergent continued expansion of the form \eqref{sfraction-n}
where $b_i>0$. The first few terms happen to be given by \eqref{cf-asymp-expansion}.
\item Thus, we conclude that the convergents of \eqref{sfraction-n} alternately provide upper and lower bounds for $r_n$. Thus, in particular, the lower bounds we have found in \S\ref{sec:lower-bounds} (as also the previous bounds of the authors mentioned) were, in principle, already available in this theory. 
\item At present, there is no closed form formula for $b_m$. 

\end{enumerate}

In view of the above, we see that Conjecture~\ref{conj2} implies Conjecture~\ref{conj1b}. In other words, assuming that our algorithm gives the same results as those of Pad\'e approximation of the Stirling series immediately leads to the conclusion that
Theorem~\ref{thm:degrees} provides a sequence of  upper and lower bounds of $r_n$. 
However, a priori, there is no reason to expect that our technique applied to a particular problem will lead to a known continued fraction.

 \section{A Ramanujan tale}\label{sec:ramanujan}
We conclude this paper with a Ramanujan story, which motivates some of the questions asked and addressed in this paper.  The story has been told by Ranganathan~\cite[p.~81]{Ranganathan1967};
it is one of P.~C.~Mahalanobis' reminiscences.
 \begin{quote}
 On another occasion, I went to his room to have lunch with him. The First World War had started sometime earlier. I had in my hand a copy of the monthly {\em Strand Magazine} which at that time used to publish a number of puzzles to be solved by readers. Ramanujan was stirring something in a pan over the fire for our lunch. I was sitting near the table, turning over the pages of the {\em Magazine}. I got interested in a problem involving a relation between two numbers. I have forgotten the details; but I remember the type of the problem. Two British officers had been billeted in Paris in two different houses in a long street; the door numbers of these houses were related in a special way; the problem was to find the two numbers. It was not at all difficult. I got the solution in a few minutes by trial and error. 
 
 {\sc Mahalanobis}: (In a joking way), Now here is a problem for you. 
  
 {\sc Ramanujan}: What problem, tell me. (He went on stirring the pan).
 
 I read out the question from the {\em Strand magazine}.
 
{\sc Ramanujan}: Please take down the solution. (He dictated a continued fraction.)

The first term was the solution which I had obtained. Each successive term represented successive solutions for the same type of relation between two numbers, as the number of houses in the street would increase indefinitely. I was amazed.

{\sc Mahalanobis}: Did you get the solution in a flash? 

{\sc Ramanujan}: Immediately I heard the problem, it was clear that the solution was obviously a continued fraction; I then thought, ``Which continued fraction?" and the answer came to my mind. It was just as simple as this.
\end{quote}

In this paper, we give a technique to answer the question ``Which continued fraction?", where the continued fraction is of the form
$$\frac{a_1}{n}\fplus\frac{a_2}{n}\fplus\frac{a_3}{n}\fplus\fdots.$$
We applied this technique to a well-known problem and explored the requirements to apply this technique. Clearly, there is much that needs to be done and 
we expect to develop this technique further in later publications. 

\subsection*{Acknowledgements}  The proofs of some of our observations are available on
 \href{https://arxiv.org/abs/2204.00962v1}{arxiv:2204.00962v1}. We thank Alan Sokal and the anonymous referees for helpful advice.


%
%

%

\end{document}